\documentclass[11pt,reqno]{amsart}
\usepackage{fullpage,enumerate}
\usepackage{amsfonts}
\usepackage{amssymb}
\usepackage{times}
\usepackage{graphicx}
\usepackage{amsmath,bm}
\usepackage{tikz-cd}
\usepackage{tikz}
\usepackage{appendix}

\usepackage{color}\usepackage{float}
\usepackage{amsmath,amssymb,latexsym,amsthm}
\usepackage{mathtools}
\usepackage[hidelinks,backref=page]{hyperref}
\hypersetup{
    colorlinks=true,
    linkcolor=blue,
    citecolor=blue,
    }

\usepackage{mathrsfs}
\usepackage{geometry}
\newtheorem{thm}{Theorem}[section]
\newtheorem{cor}[thm]{Corollary}
\newtheorem{lem}[thm]{Lemma}
\newtheorem{prop}[thm]{Proposition}

\theoremstyle{definition}
\newtheorem{ex}[thm]{Example}
\newtheorem{nota}[thm]{Notation}
\newtheorem{defn}[thm]{Definition}
\newtheorem{conv}[thm]{Convention}
\theoremstyle{remark}
\newtheorem{rem}[thm]{Remark}
\newcommand{\CR}{{\mathcal R}}
\newcommand{\BR}{{\mathbb R}}
\newcommand{\BQ}{{\mathbb Q}}
\newcommand{\BG}{{\mathbb G}}
\newcommand{\BP}{{\mathbb P}}
\newcommand{\BZ}{{\mathbb Z}}
\newcommand{\Lin}{{\mathrm{Lin}}}
\newcommand{\val}{{\mathrm{val}}}
\newcommand{\trop}{{\mathrm{Trop}}}
\newcommand{\codim}{{\mathrm{codim}}}
\newcommand{\rec}{{\mathrm{Rec}}}
\newcommand{\simb}{{\stackrel b\sim}}
\begin{document}

\address[Xiang He]{Yau Mathematical Sciences Center, Tsinghua University, Haidian District, Beijing, China, 100084}
\email{xianghe@mail.tsinghua.edu.cn}
\thanks{The author is supported by the National Key R$\&$D Program of China 2022YFA1007100 and the NSFC grant 12301057.}

\title{Positivity of tropical multidegrees}
\author{Xiang He}
\maketitle

\begin{abstract}
Let $\Gamma$ be a tropical variety in $\BR^m=\BR^{m_1}\times\cdots\times\BR^{m_k}$. We show that, under a certain condition, the positivity of the stable intersection of $\Gamma$ with certain tropical varieties pulled back from each $\BR^{m_i}$ is governed by the dimensions of the images of $\Gamma$ under all possible projections from $\BR^m$. As an application, we give a tropical proof of the criterion of the positivity of the multidegrees of a closed subscheme of a multi-projective space, carried out in the paper \cite{castillo2020multidegrees} by Castillo et al.
\end{abstract}

\section{Introduction}

Let $\BP=\BP^{m_1}\times \cdots \times\BP^{m_k}$ be a multiprojective space over an algebraically closed field, and $X\subset \BP$ a irreducible closed subscheme. For each tuple $\mathbf n=(n_1,\dotsc,n_k)\in\BZ_{\geq 0}^k$ such that $n_i\leq m_i$ and $n_1+\cdots+ n_k=\dim X$, the \textit{multidegree} $\deg^{\mathbf n}_\BP(X)$ of $X$ of \textit{type} $\mathbf n$ is defined as the intersection number of $X$ with the pullback of a linear subspace of codimension $n_i$ in each $\BP^{m_i}$.
The multidegrees are fundamental invariants of $X$ that describe the class of $X$ in the Chow ring of $\mathbb P$. The study of multidegrees dates back to the works on Hilbert functions of van der Waerden \cite{waerden1929hilbert} and Bhattacharya \cite{bhattacharya1957hilbert}; and the positivity of the multidegrees was considered, for instance, in \cite{herrmann1997reduction,huh2012milnor,katz1994hilbert,kirby1994multiplicities,trung2001positivity}, in terms of mixed multiplicities.

For each subset $I$ of $[k]$ (Notation~\ref{nt:[k]}), denote $n_I=\sum_{i\in I}n_i$. 
Recently, Castillo, Cid-Ruiz, Li, Monta\~{n}o, and Zhang gave a criterion for the positivity of $\deg^{\mathbf n}_\BP(X)$ in terms of the dimensions of the images of $X$ under the projection maps $p_I\colon\mathbb \BP\rightarrow \prod_{i\in I}\BP^{m_i}$:
\begin{thm}$($\cite[Theorem A]{castillo2020multidegrees}$)$\label{thm:castillo}
    We have $\deg^{\mathbf n}_\BP(X)>0$ if and only if $\dim p_I(X)\geq n_I$ for all $I\subset [k]$.
\end{thm}
In characteristic 0, this follows from the results of Kaveh and Khovanskii \cite{kaveh2016complete} (see also \cite[Remark 3.15]{castillo2020multidegrees}), and is later extended to the K\"ahler setting by Hu and Xiao \cite{hu2022hard}. In this paper, we consider a tropical version of the problem. 

The main objects we study in this paper are \textit{translation-admissible} tropical varieties in $\CR:=\BR^{m_1}\times\cdots\times\BR^{m_k}$, see Definition~\ref{defn:translation-admissible}. These are just tropical varieties that remain tropical varieties after taking the Minkowski sum with any rational linear subspace of $\CR$. This property should be regarded as an analogue of the irreducibility of schemes. Moreover, the tropicalization of an irreducible variety is translation-admissible (Lemma~\ref{lem:tropicalization translation-admissible}), but the opposite is not necessarily true (Remark~\ref{rem:translation-admissible not algebraic}).

Let $\Gamma\subset\CR$ be a translation-admissible tropical variety, and let $\Lambda_i\subset \BR^{m_i}$ be a fixed positive tropical divisor (as an analogue of ample divisors in algebraic geomery, see Definition~\ref{defn:positive}). Let $\mathbf n=(n_1,\dotsc,n_k)$ be as above and assume $n_1+\cdots+n_k=\dim \Gamma$. We can similarly define the \textit{multidegree} $\deg^{\mathbf n}_\CR(\Gamma)$ of $\Gamma$ of \textit{type} $\mathbf n$ and with respect to $\{\Lambda_i\}_i$ as the intersection number of $\Gamma$ with the pullback of $n_i$ copies of $\Lambda_i$ for each $i$.
Similarly as before, let $\pi_I\colon\CR\rightarrow \prod_{i\in I}\BR^{m_i}$ be the projection map. Then

\begin{thm}\label{thm:introduction} 
    We have 
$\deg^{\mathbf n}_\CR(\Gamma)>0$ if and only if $\dim\pi_I(\Gamma)\geq n_I$ for all $I\subset [k]$.
\end{thm}

See Theorem~\ref{thm:positivity tropical cycle} for a more detailed statement. Plainly, the positivity of $\deg^{\mathbf n}_\CR(\Gamma)$ is independent of the choice of $\Lambda_i$. 
The main technique of our proof is the (bounded) rational equivalence relation between tropical varieties, developed by Allerman and Rau \cite{allermann2016rational}, which is compatible with the intersection product (also known as stable intersection) between tropical varieties. This enables us to replace the intersection number in $\deg^{\mathbf n}_\CR(\Gamma)$ with the intersection number of some simpler tropical varieties, such as unions of linear subspaces of $\CR$. Combining the property of being translation-admissible, the conclusion follows from an induction on the codimension of $\Gamma$, conducted by replacing $\Gamma$ with its Minkowski sum with a certain rational line in $\CR$.

As an application, utilizing the lifting property of tropical intersections developed by Osserman and Payne \cite{osserman2013lifting}, we recover Theorem~\ref{thm:castillo} for algebraically closed fields. Moreover, we state the result in such a way that the ambient space can be any product of smooth projective varieties. See Corollary~\ref{cor:positivity in the algebraic case}.
It should be mentioned that, via the tropical lifting property, one can potentially provide more specific estimations than the positivity of the multidegrees of subschemes of $\BP$ by studying the tropical side. An example would be the log-concavity of the multidegrees discussed in \cite{huh2012milnor} for $k=2$. 
On the other hand, as translation-admissible tropical varieties are not necessarily realizable as the tropicalization of algebraic varieties (Remark~\ref{rem:translation-admissible not algebraic}), it is not clear if one can recover Theorem~\ref{thm:introduction} from Theorem~\ref{thm:castillo}. 

Finally, we would like to discuss a couple of further directions. First of all, we would like to investigate the connection between translation-admissible tropical varieties and 
tropical varieties connected through codimension 1 (\cite[Definition 3.3.4]{maclagan2021introduction}). 
Note that in Example~\ref{ex:non translation-admissible} we found a tropical variety, which is connected through codimension 1 but not translation-admissible, that violates Theorem~\ref{thm:introduction}. Secondly, for a given irreducible subscheme $X$ of $\BP$, Castillo et al. proved that the dimension $\dim p_I(X)$ is a \textit{submodular} function on $2^{[k]}$, in other words, we have 
$$\dim p_{I\cap J}(X)+\dim p_{I\cup J}(X)\leq \dim p_I(X)+\dim p_J(X)$$
for any $I,J\subset [k]$. It then follows that the support $$\mathrm{MSupp}(X)=\{\mathbf n\in\BZ_{\geq 0}^k|\deg^{\mathbf n}_\BP(X)>0\}$$
 is a \textit{polymatroid} (\cite[\S~2.3]{castillo2020multidegrees}). We prove a similar result for translation-admissible tropical varieties in Proposition~\ref{prop:submodular}. It would be interesting to study the connection between the polymatroids coming from translation-admissible tropical varieties and the linear polymatroids, as well as the algebraic polymatroids. See \cite[Theorem C]{castillo2020multidegrees} for parallel results of $\mathrm{MSupp}(X)$.

 \subsection*{Roadmap.} In Section~\ref{sec:prelim} we introduce necessary backgrounds in tropical geometry, especially tropical intersection theory, and prove some technical lemmata. In Section~\ref{sec:proof} we prove our main results.

\subsection*{Acknowledgement}
I would like to thank Naizhen Zhang for helpful conversations. I would also like to thank the anonymous referee for many insightful comments that improved the presentation.
\subsection*{Notation and Conventions}
\begin{conv}
    Throughout this paper, 
    all algebraic varieties are assumed irreducible and defined over an algebraically closed field $K$ with a non-trivial valuation $\val\colon K\rightarrow \BR$ such that $\BQ\subset \val(K)$. Suppose such a field $K$ is fixed, we use $\BG_m$ to denote the multiplicative group of non-zero elements in $K$, i.e., the algebraic torus over $K$ with dimension 1.
\end{conv}

\begin{conv}
    Let $Y,Z$ be two subvarieties of a smooth algebraic variety $X$. We say that $Y$ and $Z$ \textit{intersect properly} if $Y\cap Z$ has codimension $\codim Y+\codim Z$ in $X$. 
    In this case, let $Y\cdot Z$ denote the refined intersection cycle. 
    In addition, if $\codim Y+\codim Z=\dim X$, we denote by $\deg (Y\cdot Z)$ the sum of the intersection multiplicities of $Y$ and $Z$ along all points of $Y\cap Z$. 
\end{conv}

\begin{conv}
    A real space $\BR^m$ is always assumed to be equipped with an underlying lattice $\BZ^m$. By a  linear line (resp. linear hyperplane) in $\BR^m$ we mean a line (resp. hyperplane) that is also a linear subspace of $\BR^m$ (i.e., containing the origin).
\end{conv} 

\begin{conv}
    The projective space $\BP^m$ will be regarded as the toric variety associated to the polytope in $(\BR^m)^*$ with vertices $0,\vec e_1,\dotsc,\vec e_m$, where $\{\vec e_i\}_i$ is dual to the stand basis of $\BZ^m$.
\end{conv}

\begin{conv}
    All polyhedral complexes $\Gamma$ in $\BR^m$ are assumed rational and pure-dimensional, unless otherwise specified. The maximal faces of $\Gamma$ are called \textit{facets}.   
\end{conv}

\begin{nota}
    For two subsets $A,B$ of $\BR^m$, we denote their Minkowski sum by $A+B$.
\end{nota}

\begin{nota}
    Let $P\subset \BR^m$ be a polyhedron. We denote by $\Lin(P)$ the linear subspace of $\BR^m$ parallel to $P$ and has dimension equal to $P$. Denote by $\rec(P)$ the \textit{recession cone} of $P$:
    $$\rec(P)=\{\vec v\in\BR^m|x+\BR_{\geq 0}\vec v\subset P \mathrm{\ for\ all\ }x\in P.\}$$
\end{nota}

\begin{nota}\label{nt:[k]}
    Let $k$ be a positive integer. Denote $[k]=\{1,\dotsc,k\}$.
\end{nota}

\section{Preliminaries}\label{sec:prelim}

\subsection{Tropical intersection theory} In this subsection, we briefly recall some basic terminology from tropical intersection theory. The main reference is the paper \cite{allermann2016rational} of Allermann, Hampe, and Rau, developed for the ambient space $\BR^m$. See also \cite{he2019generalization,meyer2011intersection} for related works on compactified $\BR^m$. The $\BR^m$ case is enough for our later applications. 

We start with the notion of tropical cycles. Let $\Gamma\subset \BR^m$ be a  rational, finite, and pure-dimensional polyhedral complex. Let $\omega$ be a \textit{weight} on $\Gamma$ which assigns an integer to each facet of $\Gamma$.

\begin{defn}(\cite[\S~2.1]{allermann2016rational})\label{defn:tropical cycle}
     We say that $\Gamma$ is a \textit{tropical polyhedral complex} if for each codimension 1 face $Q$ of $\Gamma$, the following \textit{balancing condition} at $Q$ is satisfied:
     \begin{equation}\label{eq:balancing}
         \sum_{Q\subset P}\omega(P)\vec v_{P/Q}=0.
     \end{equation}
     Here, $P$ runs over all facets of $\Gamma$ containing $Q$; and $\vec v_{P/Q}$ is the primitive integer generator of the ray obtained from projecting $P$ to $\BR^m/\Lin(Q)$. 

     A \textit{tropical cycle} is an equivalent class of tropical polyhedral complexes, defined up to refinement. The \textit{support} of a tropical cycle is the union of facets with non-zero weights.
\end{defn}

\begin{conv}
    In the sequel, for simplicity, we will only consider tropical cycles with non-negative weights. Following the conventions in \cite{allermann2010first}, we call such tropical cycles \textit{tropical varieties}. 
\end{conv}

To reduce notation, we will often denote a tropical variety also by $\Gamma$ and assume that a tropical polyhedral complex structure (with non-negative weights) is chosen. We will also use $\Gamma$ to represent its support, if there is no confusion. 
 A tropical variety of codimension 1 in $\BR^m$ will be called a \textit{tropical divisor}; similarly, a dimension-1 tropical variety will be called a \textit{tropical curve}. If $\dim\Gamma=0$, we call the sum of all weights of the points in $\Gamma$ the \textit{degree} of $\Gamma$ and denote by $\deg\Gamma$. In this case, we will simply write $\Gamma>0$ if $\deg\Gamma>0$, in other words, if $\Gamma$ is non-empty.

\begin{defn}(\cite[Definition 5.1]{allermann2016rational})
    Let $\Gamma$ be a tropical variety, and denote $$\rec(\Gamma)=\{\rec(P)|P\mathrm{\ is\ a\ face\ of\ }\Gamma\}.$$
    Then, up to a suitable refinement, $\rec(\Gamma)$ becomes a tropical variety (with induced weight from $\Gamma$) which is a fan. This is called the \textit{recession cycle} of $\Gamma$.
\end{defn}

Similar to algebro-geometric intersection theory, a \textit{bounded rational equivalence} relation between tropical varieties was developed in \cite[\S 3]{allermann2016rational}. 
It was denoted by ``$\simb$." 
We will not need the exact construction of bounded rational equivalence; however, the following proposition is essential:

\begin{prop}\label{prop:basic property tropical intersection}
    Let $\Gamma,\Gamma'\subset \BR^m$ be tropical varieties. Then we have:
    \begin{enumerate}
    \item $\mathrm{(}$\cite[Proposition 3.2 (v)]{allermann2016rational}$\mathrm{)}$ If $\Gamma\simb \Gamma'$ and $\dim \Gamma=\dim\Gamma'=0$, then $\deg \Gamma=\deg\Gamma'$.
        \item $\mathrm{(}$\cite[Proposition 3.3]{allermann2016rational}$\mathrm{)}$ For any $\vec v\in\BR^m$, the translation $\Gamma+\vec v$ is also a tropical variety and $\Gamma\simb \Gamma+\vec v$.
        \item $\mathrm{(}$\cite[Theorem 5.3]{allermann2016rational}$\mathrm{)}$ We have 
        $\Gamma\simb\rec(\Gamma)$.
    \end{enumerate}
\end{prop}

For two tropical varieties $\Gamma$ and $\Gamma'$ in $\BR^m$, one can perform the \textit{stable intersection} $\Gamma\cdot\Gamma'$, which is an analogue of the intersection product in algebraic geometry. It is a tropical variety with codimension equal to $\codim\Gamma+\codim\Gamma'$, and supported on the union of intersections of faces $P$ and $P'$ of $\Gamma$ and $\Gamma'$, respectively, such that $\dim( P+P')=m$. We refer the readers to \cite[\S~3.6]{maclagan2021introduction} for a concrete construction of the stable intersection (including the weights of its facets). 
The following proposition says that stable intersection is compatible with bounded rational equivalence.

\begin{prop}$\mathrm{(}$\cite[Proposition 3.2 (iv)]{allermann2016rational}$\mathrm{)}$
    Let $\Gamma,\Gamma',\Gamma''$ be tropical varieties in $\BR^m$. If $\Gamma'\simb \Gamma''$, then $\Gamma\cdot\Gamma'\simb \Gamma\cdot\Gamma''$. 
\end{prop}

We end this subsection with the notion of the transverse intersection of tropical varieties.

\begin{defn}(\cite[Definition 3.4.9]{maclagan2021introduction}])
    Let $\Gamma,\Gamma'$ be tropical varieties in $\BR^m$ and $x\in \Gamma\cap\Gamma'$. We say that $\Gamma$ and $\Gamma'$ \textit{intersect transversely} at $x$ if there are polyhedral complex structures on $\Gamma$ and $\Gamma'$ such that, if $P$ (resp. $P'$) is the unique open face of $\Gamma$ (resp. $\Gamma'$) containing $x$, then $\Lin(P)+\Lin(P')=\BR^m$. We say that $\Gamma$ and $\Gamma'$ \textit{intersect transversely} if it is so at all points of $\Gamma\cap \Gamma'$.
\end{defn}

Clearly, if $\Gamma$ intersects $\Gamma'$ transversely, then $\Gamma\cdot\Gamma'$ is supported on $\Gamma\cap \Gamma'$.

\subsection{Positive and translation-admissible tropical varieties}
In this subsection, we introduce the main properties of a tropical variety that we want to study.
\begin{defn}\label{defn:positive}
    A tropical divisor $\Gamma\subset\BR^m$ is \textit{positive} if for any non-empty tropical curve $\Gamma'$ we have $\Gamma\cdot\Gamma'>0$. 
\end{defn}

\begin{defn}\label{defn:translation-admissible}
    A tropical variety $\Gamma\subset\BR^m$ is \textit{translation-admissible} if for any rational linear subspace $V$ of $\BR^m$, the Minkowski sum $\Gamma+V$ is supported on a tropical variety.
\end{defn}

Note that the Minkowski sum $\Gamma+V$ is the image of $\Gamma\times V$ under the linear projection given by $(x,y)\mapsto x+y$, and  $\Gamma\times V$ is naturally a tropical variety. According to \cite[Lemma 2.2]{jensen2016stable}, $\Gamma+V$ is supported on a tropical variety if and only if it has pure dimension. Hence $\Gamma$ is translation-admissible if and only if $\Gamma+V$ is pure dimensional for any rational linear subspace $V$ of $\BR^m$.
Clearly, any linear subspace of $\BR^m$ is translation-admissible. The tropicalization of any subvariety of an algebraic torus is also translation-admissible (Lemma~\ref{lem:tropicalization translation-admissible}). On the other hand, translation-admissible tropical varieties do not have to be connected. For example, the union of two parallel lines is translation-admissible.

The property of being translation-admissible enables us to do induction on the codimension of a tropical variety in the proof of our main result. See Theorem~\ref{thm:positivity tropical cycle} for details. We next prove some technical lemmata for future use.
\begin{lem}\label{lem:positive stable intersection}
    Let $\Gamma,\Gamma'$ be tropical varieties in $\BR^m$ with complementary codimensions. Then the following are equivalent:
    \begin{enumerate}
        \item $\Gamma\cdot\Gamma'>0$;
         \item there is a facet $P'$ of $\Gamma'$ such that $\Gamma\cdot\Lin (P')>0$;
        \item there are facets $P$ and $P'$ of $\Gamma$ and $\Gamma'$ respectively, such  that $\Lin (P)+\Lin (P')=\BR^m$.
    \end{enumerate}
\end{lem}
\begin{proof}
    Up to translation, we may assume that $\Gamma$ intersects  $\Gamma'$ transversely. In this case, the stable intersection $\Gamma\cdot \Gamma'$ is supported on $\Gamma\cap\Gamma'$. Hence $\Gamma\cdot\Gamma'>0$ if and only if $\Gamma\cap\Gamma'$ is nonempty. This shows (1)$\Rightarrow$(3).
For (3)$\Rightarrow$(2), we again translate $\Gamma$ such that $P\cap P'$ is nonempty and transverse. Then  $\Gamma\cdot\Lin(P')$ contains $P\cap P'$. Hence $\Gamma\cdot\Lin (P')>0$.
    For (2)$\Rightarrow$(1), we translate $\Gamma$ so that it intersects $P'$ transversely at some point. Hence $\Gamma\cdot\Gamma'>0$.
\end{proof}
\begin{rem}\label{rem:positive divisors}
    According to Lemma~\ref{lem:positive stable intersection}, a tropical divisor $\Gamma\subset \BR^m$ is positive if and only if for every linear line $L$ in $\BR^m$, there is a facet $P$ of $\Gamma$ such that $L$ is not contained in $\Lin(P)$. In particular, if $\Gamma\subset\BR^m$ be a union of linear hyperplanes, then $\Gamma$ is positive if and only if there are hyperplanes $\{H_i\}_{i\in [m]}$ in $\Gamma$ such that $\cap_{i\in[m]}H_i=\{0\}$.
\end{rem}

\begin{lem}\label{lem:projection of translation-admissible tropical cycle}
    Suppose $\Gamma\subset\BR^m$ is a translation-admissible tropical variety. Suppose we have a decomposition $\BR^m=\BR^{m_1}\times\BR^{m_2}$. Then the image of $\Gamma$ in $\BR^{m_i}$ under projection is a translation-admissible tropical variety in $\BR^{m_i}$.
\end{lem}
\begin{proof}
    By symmetry, we only prove the conclusion for $\BR^{m_1}$. Let $\pi$ be the projection from $\BR^m$ to $\BR^{m_1}$. Let $\Gamma'=\Gamma+\BR^{m_2}$. Then $\Gamma'$ is also a translation-admissible tropical variety, and $\pi(\Gamma)=\pi(\Gamma')$. Since $\Gamma'=\pi(\Gamma')\times \BR^{m_2}$ and $\Gamma'$ is translation-admissible, $\pi(\Gamma')$ is a translation-admissible tropical variety in $\BR^{m_1}$. This proves the lemma.
\end{proof}

\begin{lem}\label{lem:facet translation invariant}
    Let $\Gamma$ be a tropical variety in $\BR^m$ and $\vec v\in \BQ^m$ a vector. Then $\Gamma+\BR\vec v=\Gamma$ if and only if  $\vec v\in\Lin (P)$ for any facet $P$ of  $\Gamma$.  In particular, if $\Gamma\subsetneq \Gamma+\BR\vec v$ then $\dim (\Gamma+\BR\vec v)=\dim \Gamma+1$; if furthermore that $\Gamma$ is translation-admissible, then $\Gamma+\BR\vec v$ has pure dimension $\dim\Gamma+1$.
\end{lem}
\begin{proof}
    If $\Gamma+\BR\vec v=\Gamma$, then for every facet $P$ of $\Gamma$ we have 
    $\dim(P+\BR\vec v)\leq \dim (\Gamma+\BR\vec v)=\dim \Gamma=\dim P$. Hence $\vec v\in\Lin(P)$.

    Now suppose $\vec v\in\Lin (P)$ for any facet $P$ of $\Gamma$. We want to show that $\Gamma= \Gamma+\BR\vec v$. Let $\Lambda$ be the union of all codimension-1 faces $Q$ of $\Gamma$ such that $\vec v\in\Lin (Q)$. Then $\dim(\Lambda+\BR\vec v)=\dim\Lambda=  \dim\Gamma-1$. It suffices to show that for each $x\in\Gamma\backslash(\Lambda+\BR\vec v)$ we have $x+\BR \vec v\subset \Gamma$. 

    Suppose there is an $x\in\Gamma\backslash(\Lambda+\BR\vec v)$ such that $x+\BR \vec v\not\subset \Gamma$. 
    Since $\Gamma\cap (x+\BR \vec v)$ is a union of closed intervals, there is a $y\in (x+\BR \vec v)\cap \Gamma$ and a positive real number $\epsilon$ such that  $y+a\vec v\not\in\Gamma$ either for any $0<a<\epsilon$ or  for any $-\epsilon<a<0$. Suppose the latter situation happens. By the assumption on $x$, the point $y$ is contained in a codimension 1 face $Q$ of $\Gamma$ such that $\vec v\not\in\Lin(Q)$. 
    Then for each facet $P$ of $\Gamma$ containing $Q$, the vector $\vec v_{P/Q}$ is represented by a positive multiple of $\vec v$. This contradicts with the balancing condition~(\ref{eq:balancing}) of $\Gamma$ at $Q$.
\end{proof}

\subsection{The tropicalization map}
Let $T=(K^*)^m$ be the algebraic torus of dimension $m$. The tropicalization map $\trop\colon T\rightarrow \BR^m$ is defined as 
$$x=(x_1,\dotsc,x_m)\mapsto \trop(x)=(\val x_1,\dotsc,\val x_m).$$
For a (irreducible) subvariety $X$ of $T$, we denote by $\trop(X)$ the closure of the set $\{\trop(x)|x\in X\}$. If $X$ is closed in $T$, then $\trop(X)$ admits a tropical variety structure inherited from the scheme structure of $X$ and $\dim(\trop(X))=\dim X$. See \cite[\S~3]{maclagan2021introduction} for details. If $X$ is not necessarily closed, by construction, $\trop(X)$ and $\trop(\overline X)$ have the same support. 

\begin{lem}\label{lem:tropicalization translation-admissible}
    Let $X$ be a closed subvariety of $T$. Then $\trop(X)$ is a translation-admissible tropical variety.
\end{lem}
\begin{proof}
   Let $\vec v=(a_1,\dotsc,a_m)\subset \BZ^m$ be an integral vector. We claim that $\trop(X)+\BR\vec v$ is also the tropicalization of a closed subvariety of $T$, hence a tropical variety. Since any rational linear subspace of $\BR^m$ is generated by finitely many integral vectors, by induction, we are done.
   
   Indeed, let $\mathbb G_m\rightarrow T$ be the cocharacter of $T$ given by $t\mapsto (t^{a_1},\dotsc,t^{a_m})$.
   By Chevalley's Theorem, $X':=\BG_m\cdot X$ is a constructible and irreducible subset of $T$. Hence $\overline X'$ is a closed   subvariety of $T$ and
   $\trop(\overline X')=\trop(X')=\trop(X)+\BR\vec v$. This proves the claim. 
\end{proof}

\begin{rem}\label{rem:translation-admissible not algebraic}
    According to \cite[Theorem 3.5.1]{maclagan2021introduction}, $\trop(X)$ is connected through codimension 1. However, not all translation-admissible tropical varieties (even the ones connected through codimension 1) come from subvarieties of $T$. For example, all connected tropical curves in $\BR^2$ are translation-admissible and connected through codimension 1; however, only the ``well-spaced" ones are realizable. See \cite{katz2012lifting,speyer2014parametrize}. 
\end{rem}

Whether the tropicalization map commutes with the intersection of two subvarities $X,X'$ of $T$ is an important question in tropical geometry. A positive answer was provided by Bogart et al. \cite{bogart2007computing} when $\trop(X)$ and $\trop(X')$ intersect transversely. Later, Osserman and Payne \cite{osserman2013lifting} generalized the result to the proper-intersection case, meaning that, as in the algebro-geometric case, the intersection $\trop(X)\cap\trop(X')$ has codimension as expected everywhere. Plainly,  transverse intersection implies proper intersection. Moreover, they also proved a lifting theorem for the intersection multiplicities. We recall part of their results for our later use.

\begin{thm}$\mathrm{(}$\cite[Corollary 5.1.2]{osserman2013lifting}$\mathrm{)}$\label{thm:lifting tropical intersections}
    Let $X,X'\subset T$ be closed subvarieties. Suppose $\trop(X)$ and $\trop(X')$ meet properly. Then $X$ intersects $X'$ properly and $$\trop(X\cdot X')=\trop(X)\cdot \trop(X)$$ as tropical varieties. In particular, if $\codim X+\codim X'=m$, then $$\deg (X\cdot X')=\deg( \trop(X)\cdot \trop(X')).$$
\end{thm}

Here $X\cdot X'$ represents the refined intersection cycle. See also \cite{he2019generalization,osserman2011lifting} for similar results regarding the tropicalization of (partially) compactified tori, and \cite{he2018lifting} for the Berkovich-analytic setting. We end this section with a corollary about the positivity of tropicalizations of hyperplanes.

\begin{cor}\label{cor:tropical hyperplane positive}
    Suppose $H$ is a hyperplane in $\BP^m$ intersecting transversely with all torus orbits. Let $T\subset \BP^n$ be the maximal torus and $H^\circ=H\cap T$. Then $\trop(H^\circ)$ is positive in $\BR^m$. 
\end{cor}
\begin{proof}
    According to Remark~\ref{rem:positive divisors}, it suffices to show that for every rational line in $\BR^m$, there is a facet of $\trop(H^\circ)$ not parallel to it. Suppose $T=\mathrm{Spec}K[x_1^{\pm},\dotsc,x_m^{\pm}]$. Then $H^\circ$ is given by an equation $\lambda_0+\lambda_1x_1+\cdots+\lambda_mx_m=0$ such that $\lambda_i\neq 0$ for all $i$. It follows that 
    $\trop(H^0)$ is the set of $(a_1,\dotsc,a_m)\in\BR^m$ such that the minimum
    $$\min\{\val(\lambda_0),\val(\lambda_1)+a_1,\cdots, \val(\lambda_m)+a_m\}$$
    is obtained at least twice. It is then easy to check that $\trop(H^\circ)$ contains a facet parallel to each coordinate hyperplane of $\BR^m$, hence it is positive.    
\end{proof}

\section{Proof of main results}\label{sec:proof}
In this section, we prove Theorem~\ref{thm:introduction}, which is re-interpreted in Theorem~\ref{thm:positivity tropical cycle}, and discuss a couple of applications.

Let $m,n,k$ be positive integers such that $m>n$. Fix a tuple of positive integers $(m_1,\dotsc,m_k)$ and a tuple of non-negative integers $(n_1,\dotsc,n_k)$ such that $m_i\geq n_i$. Furthermore, suppose that $\sum_{i\in[k]}m_i=m$ and $\sum_{i\in [k]}n_i=n$. For each subset $I$ of $[k]$, denote $n_I=\sum_{i\in I}n_i$.

Let $\CR=\BR^m=\prod_{i\in [k]}\BR^{m_i}$. For $I\subset [k]$, denote similarly $\CR_I=\prod_{i\in I}\BR^{m_i}$. Let $\pi_I\colon \CR\rightarrow \CR_I$ be the projection map. To reduce notation, for each $i\in[k]$, we write $\pi_{\{i\}}=\pi_i$. We also consider $\CR_I$ as a subspace of $\CR$ by setting the coordinates irrelevant to $I$ to be $0$.

\begin{lem}\label{lem:existence of face}
    Let $\Gamma\subset\CR$ be a translation-admissible tropical variety with dimension $n$. Suppose for any $I\subset [k]$, we have $\dim\pi_I(\Gamma)\geq n_I$. Then there is a facet $Q$ of $\Gamma$ such that $\dim\pi_I(Q)\geq n_I$ for any $I\subset [k]$. 
\end{lem}
\begin{proof}
    We proceed by induction on the codimension of $\Gamma$. The base case, where $\dim \Gamma=m$, is trivial. We may then assume $n<m$. Without loss of generality, suppose $n_k<m_k$. For any $I\subset [k]$ and any facet $P$ of $\pi_I(\Gamma)$ such that $\BR^{m_k}\not\subset\Lin(P)$, the intersection $\BR^{m_k}\cap\Lin(P)$ has codimension at least one in $\BR^{m_k}$. We pick a vector $\vec u\in\BR^{m_k}$ not contained in any such intersection.   

   If $\BR^{m_k}\subset \Lin(P)$ for every facet $P$ of $\Gamma$, then $\Gamma+\BR\vec v=\Gamma$ for all $\vec v\in \BR^{m_k}$ by Lemma~\ref{lem:facet translation invariant}. Hence $\Gamma=\Gamma'\times \BR^{m_k}$ for some tropical variety $\Gamma'$ in $\CR_{[k-1]}$. In this case, $\dim\pi_{[k-1]}(\Gamma)=\dim\Gamma'=n-m_k<n_{[k-1]}$, which contradicts the assumption. Therefore, for some facet $P$ of $\Gamma$ we have $\BR^{m_k}\not\subset \Lin(P)$. Thus $\vec u\not\in \Lin(P)$ by construction. Let $\Gamma_{\vec u}=\Gamma+\BR\vec u$. Since $\Gamma$ is translation-admissible, so is $\Gamma_{\vec u}$. Again, by Lemma~\ref{lem:facet translation invariant}, $\dim\Gamma_{\vec u}=\dim \Gamma+1$. We will utilize the inductive assumption on $\Gamma_{\vec u}$. 

   First of all, let $I$ be a subset of $[k-1]$. We then have $\dim \pi_I(\Gamma_{\vec u})= \dim \pi_I(\Gamma)\geq n_I$. Secondly, let $J=I\cup\{k\}$ be a subset of $[k]$ containing $k$. Then $\pi_J(\Gamma_{\vec u})=\pi_J(\Gamma)+\BR\vec u$. By the construction of $\vec u$, either 
   \begin{enumerate}
       \item [(i)] $\BR^{m_k}\subset \Lin(P)$ for any facet $P$ of $\pi_J(\Gamma)$; or
       \item [(ii)] for some facet $P$ of $\pi_J(\Gamma)$ we have $\vec u\not\in \Lin(P)$.
   \end{enumerate}   
   Note that $\pi_J(\Gamma)$ is also a tropical variety by Lemma~\ref{lem:projection of translation-admissible tropical cycle}. 
   In case (i), we have $\pi_{J}(\Gamma)=\Gamma'\times \BR^{m_k}$ for some tropical variety $\Gamma'\subset \CR_I$ such that $\dim\Gamma'=\dim\pi_I\pi_J(\Gamma)=\dim\pi_I(\Gamma)\geq n_I$. Hence $$\dim\pi_J(\Gamma_{\vec u})\geq \dim\pi_J(\Gamma)=\dim \Gamma'+m_k\geq n_I+m_k\geq n_I+(n_k+1).$$
   In case (ii), by Lemma~\ref{lem:facet translation invariant} we still have $$\dim\pi_J(\Gamma_{\vec u})=\dim\pi_J(\Gamma)+1\geq n_J+1=n_I+(n_k+1).$$
   To sum up, let $n'_i=n_i$ for $i\in[k-1]$ and $n'_k=n_k+1$. Let $n'_I=\sum_{i\in I}n'_i$ as usual. Then $\sum_{i\in [k]} n'_i=n+1=\dim\Gamma_{\vec u}$, and $\dim \pi_I(\Gamma_{\vec u})\geq n'_I$ for any $I\subset [k]$.

   By induction, there is a facet $Q'=Q+\BR\vec u$ of $\Gamma_{\vec u}$, where $Q$ is a facet of $\Gamma$, such that $\dim\pi_I(Q')\geq n'_I$ for any $I\subset [k]$. Now if $k\not\in I$, then $\dim\pi_I(Q)=\dim\pi_I(Q')\geq n'_I=n_I$; if $k\in I$, then $\dim\pi_I(Q)\geq\dim\pi_I(Q')-1\geq n'_I-1=n_I$. Hence $Q$ is the desired face of $\Gamma$.
\end{proof}

Fix positive tropical divisors $\Lambda_i\subset \BR^{m_i}$. Let $\Lambda_i^{n_i}$ be the stable intersection of $n_i$ copies of $\Lambda_i$ in $\BR^{m_i}$ and $\prod_{i\in[k]}\Lambda_i^{n_i}$ be their direct product in $\CR$. Note that there is a natural tropical variety structure on each $\pi_i^{-1}(\Lambda_i)$ induced from $\Lambda_i$, and    $\prod_{i\in[k]}\Lambda_i^{n_i}$ is just the stable intersection of $n_i$ copies of $\pi_i^{-1}(\Lambda_i)$ for each $i\in [k]$. Let $\Gamma\subset \CR$ be a translation-admissible tropical variety of dimension $n$. Recall that $\deg^{\mathbf n}_\CR(\Gamma)$ denotes the degree of $\Gamma\cdot\prod_{i\in [k]}\Lambda_i^{n_i}$, where $\mathbf n=(n_1,\dotsc,n_k)$. We re-state Theorem~\ref{thm:introduction} in  the following form:

\begin{thm}\label{thm:positivity tropical cycle}
     Let $\Gamma\subset \CR$ be a translation-admissible tropical variety of dimension $n$. 
    Then $\Gamma\cdot\prod_{i\in[k]}\Lambda_i^{n_i}>0$ if and only if $\dim\pi_I(\Gamma)\geq n_I$ for all $I\subset [k]$.
\end{thm}

\begin{ex}\label{ex:non translation-admissible}
We provide a couple of examples of tropical varieties that are not translation-admissible and fail Theorem~\ref{thm:positivity tropical cycle}. Note that the second one is connected through codimension 1, while the first one is not.

    Consider firstly the case $k=2,\ m_1=m_2=2,$ and $n_1=n_2=1$. Let $\vec e_1,\vec e_2$ be a basis of $\BR^{m_1}$ and $\vec e_3,\vec e_4$ be a basis of $\BR^{m_2}$. 
    Let $\Gamma$ be the union of two planes $P_1$ spanned by $\vec e_1,\vec e_2$ and $P_2$ by $\vec e_3,\vec e_4$. Then $\Gamma$ is not translation-admissible as $\Gamma+\BR\vec e_4$ is not pure-dimensional. Since $P_1\cap P_2$ is the origin, $\Gamma$ is not connected through codimension 1, either. 
    Let $\Lambda_i\subset\BR^{m_i}$ be an arbitrary positive tropical divisor. Clearly, $\dim\Gamma=2=n_1+n_2$ and $\dim\pi_i(\Gamma)=2>n_i$ for each $i=1,2$. On the other hand, since $\pi_2(P_1)$ is a single point, $P_1$ is disjoint with a general translation of $\pi_2^{-1}(\Lambda_2)$. Therefore, $P_1\cdot (\Lambda_1\times\Lambda_2)=0$. Similarly, $P_2\cdot (\Lambda_1\times\Lambda_2)=0$. Hence $\Gamma\cdot (\Lambda_1\times\Lambda_2)=0$. Therefore, $\Gamma$ violates Theorem~\ref{thm:positivity tropical cycle}.

    Now consider $k=4$, $m_1=2,m_2=m_3=1$, and $n_1=n_3=1, n_2=0$. Pick a basis $\vec e_1,\vec e_2$ of $\BR^{m_1}$, and $\vec e_3,\vec e_4$ of $\BR^{m_2}$ and $\BR^{m_3}$, respectively. Let $Q_1\subset \BR^3=\BR^{m_1}\times \BR^{m_2}$ be the standard tropical plane. In other words, $Q_1$ is the union of the following $6$ facets:
    $$\BR_{\geq 0}\vec e_1+\BR_{\geq 0}\vec e_2,\ \BR_{\geq 0}\vec e_2+\BR_{\geq 0}\vec e_3, \ \BR_{\geq 0}\vec e_1+\BR_{\geq 0}\vec e_3,$$ and $$\BR_{\geq 0}\vec e_i+\BR_{\geq 0}(-\vec e_1-\vec e_2-\vec e_3), \mathrm{\ where\ } 1\leq i\leq 3.$$ Let $Q_2=\BR\vec e_3+\BR\vec e_4$. Then $Q_1\cap Q_2=\BR_{\geq 0}\vec e_3$, and hence the tropical variety $\Gamma:=Q_1\cup Q_2$ is connected through codimension $1$. However, $\Gamma$ is not translation-admissible, as $\Gamma+\BR\vec e_3$ is not pure-dimensional. Let $\Lambda_i\subset \BR^{m_i}$ be an arbitrary positive tropical divisor, where $i=1,3$. By construction, we have $\dim\pi_1(\Gamma)=2>n_1$, $\dim\pi_3(\Gamma)=1=n_3$ and $\dim\pi_{\{1,3\}}(\Gamma)=2=n_1+n_3$. We claim that $\Gamma\cdot (\Lambda_1\times \BR^{m_2} \times \Lambda_3)=0$, hence $\Gamma$ fails  Theorem~\ref{thm:positivity tropical cycle}. 
    Indeed, since $\dim \Lambda_3=0$, the tropical variety $\Lambda_1\times \BR^{m_2} \times \Lambda_3$ is parallel to $\BR^{m_1}\times \BR^{m_2}$, hence has trivial intersection number with $Q_1$. On the other hand, note that $\pi_1(Q_2)$ is the origin in $\BR^{m_1}$. Hence $Q_2$ is disjoint with a general translation of $\pi_1^{-1}(\Lambda_1)$, and therefore has trivial intersection number with $\Lambda_1\times \BR^{m_2} \times \Lambda_3$. This proves the claim.
\end{ex}

\begin{proof}[Proof of Theorem~\ref{thm:positivity tropical cycle}]
    Suppose $\dim\pi_I(\Gamma)< n_I$ for some $I\subset [k]$. Then for a general $\vec v\in \CR_I$, the tropical variety $\pi_I(\Gamma)+\vec v$ is disjoint with $\prod_{i\in I}\Lambda_i^{n_i}$. Hence $\Gamma+\vec v$ is disjoint with $\prod_{i\in [k]}\Lambda_i^{n_i}$ and we have 
$$ \Gamma\cdot\prod_{i\in[k]}\Lambda_i^{n_i}=\Big(\Gamma+\vec v\Big)\cdot\prod_{i\in[k]}\Lambda_i^{n_i}=0.$$

    We now assume that $\dim\pi_I(\Gamma)\geq n_I$ for all $I\subset [k]$ and prove that the stable intersection is positive. By Lemma~\ref{lem:existence of face}, there is a facet $P_\Gamma$ of $\Gamma$ such that $\dim\pi_I(P_\Gamma)\geq n_I$ for all $I\subset [k]$. According to Lemma~\ref{lem:positive stable intersection}, we may replace $\Gamma$ with $\Lin(P_\Gamma)$ and thus assume that $\Gamma$ is a linear subspace of $\CR$. 
    On the other hand, after replacing each $\Lambda_i$ with its recession cycle if necessary, we may assume that $\Lambda_i$ is a fan for each $i$. Moreover, by Lemma~\ref{lem:positive stable intersection}, we may assume that for each face $P$ of $\Lambda_i$, we have $\Lin (P)\subset \Lambda_i$. In other words, $\Lambda_i$ is a union of linear hyperplanes in $\BR^{m_i}$. Lastly, by Remark~\ref{rem:positive divisors}, we may assume that $\Lambda_i$ is a union of $m_i$ linear hyperplanes in $\BR^{m_i}$ whose intersection is trivial.

    As in the proof of Lemma~\ref{lem:existence of face}, we proceed by induction on the codimension $m-n$ of $\Gamma$. If $n=m$, then $n_i=m_i$ for each $i$. Since $\Lambda_i^{m_i}>0$, we have  $\Gamma\cdot\prod_{i\in[k]}\Lambda_i^{n_i}=\prod_{i\in[k]}\Lambda_i^{m_i}>0.$ 
    
    We now assume $n<m$. Without loss of generality, suppose $n_k<m_k$. A similar argument as in Lemma~\ref{lem:existence of face}, combined with the inductive hypothesis, shows that there is a vector $\vec u\in \BR^{m_k}$ such that $\Gamma_{\vec u}:=\Gamma+\BR\vec u$ has dimension $n+1$, and $$\Gamma_{\vec u}\cdot\Big(\Lambda_k^{n_k+1}\times\prod_{i\in[k-1]}\Lambda_i^{n_i}\Big)>0.$$ By Lemma~\ref{lem:positive stable intersection}, there are facets $P_i$ of $\Lambda_i^{n_i}$ for $i\in[k-1]$ and $P_k$ of $\Lambda_k^{n_k+1}$ such that $$\Gamma+\BR\vec u+\sum_{i\in[k]}\Lin (P_i)=\Gamma_{\vec u}+\sum_{i\in[k]}\Lin (P_i)=\CR.$$
    Here $\sum_{i\in[k]}\Lin (P_i)$ is just the linear subspace of $\CR$ parallel to the facet $\prod_{i\in [k]}P_i$ of $\Lambda_k^{n_k+1}\times\prod_{i\in[k-1]}\Lambda_i^{n_i}$.
    Similarly, let $P$ be a facet of $\Lambda_k^{n_k}$. Then $P\times \prod_{i\in[k-1]}P_i$ is a facet of $\prod_{i\in [k]}\Lambda_i^{n_i}$, and it is parallel to $\sum_{i\in[k-1]}\Lin (P_i)+\Lin(P)$.
    Let us write $$L(P):=\Gamma+\sum_{i\in[k-1]}\Lin (P_i)+\Lin(P)$$ for all facet $P$ of $\Lambda_k^{n_k}$. By Lemma~\ref{lem:positive stable intersection}, it suffices to show  
    that there is a such $P$ with $L(P)=\CR$.

    By assumption, $\Lambda_k$ is the union of $m_k$ linear hyperplanes $\{H_i\}_{1\leq i\leq m_k}$ in $\BR^{m_k}$ which intersect trivially. Then there are $n_k+1$ such hyperplanes, say $H_1,\dotsc,H_{n_k+1}$, such that $P_k$ is a facet of $\bigcap_{1\leq i\leq n_k+1}H_i$. On the other hand,
    there are $n_k+1$ facets $Q_1,\dotsc, Q_{n_k+1}$ of $\Lambda_k^{n_k}$ containing $P_k$, where each $Q_j$ is a facet of the intersection of $n_k$ hyperplanes in $H_1,\dotsc, H_{n_k+1}$. We claim that there must be a $Q_j$ such that $L(Q_j)=\CR$. Note that $L(Q_j)$ contains $\Gamma+\sum_{i\in [k]}\Lin(P_i)$, which has dimension $m-1$. If $L(Q_j)\neq \CR$ for each $j$, then $L(Q_j)=\Gamma+\sum_{i\in [k]}\Lin(P_i)$. Hence
    $$\Gamma+\sum_{i\in [k]}\Lin(P_i)+\sum_{j\in[n_k+1]}\Lin(Q_j)=\Gamma+\sum_{i\in [k]}\Lin(P_i)\subsetneq \CR.$$
    On the other hand, it is easy to check that $\sum_{j\in [n_k+1]}\Lin (Q_j)=\BR^{m_k}$, which contains $\vec u$. Hence the left hand of the formula contains $\Gamma+\BR\vec u+\sum_{i\in[k]}\Lin (P_i)=\CR.$ This is a contradiction, and we are done. 
\end{proof}

Let $\mathrm{MSupp}(\Gamma):=\{\mathbf n=(n_1,\dotsc,n_k)\in\BZ^k_{\geq 0}|\deg^{\mathbf n}_\CR(\Gamma)>0\}$. By Theorem~\ref{thm:positivity tropical cycle},
$$\mathrm{MSupp}(\Gamma)=\{(n_1,\dotsc,n_k)\in\BZ^k_{\geq 0}\big|\sum_{i\in [k]}n_i=\dim\Gamma \mathrm{\ and\ } \sum_{i\in I}n_i\leq \dim\pi_I(\Gamma)\mathrm{\ for\ all\ }I\subset [k]\}.$$
The following proposition then shows that $\mathrm{MSupp}(\Gamma)$ is a polymatroid (\cite[Definition 2.14]{castillo2020multidegrees}).

\begin{prop}\label{prop:submodular}
    Let $\Gamma\subset \CR$ be a translation-admissible tropical variety of dimension $n$. For any $I,J\subset [k]$, we have 
    $$\dim \pi_{I\cap J}(\Gamma)+\dim \pi_{I\cup J}(\Gamma)\leq \dim \pi_I(\Gamma)+\dim \pi_J(\Gamma).$$
\end{prop}
\begin{proof}
    By Lemma~\ref{lem:projection of translation-admissible tropical cycle}, $\pi_{I\cup J}(\Gamma)$ is still translation-admissible. Therefore, we may assume $I\cup J=[k]$ and hence $\pi_{I\cup J}(\Gamma)=\Gamma$. The inequality then becomes
     \begin{equation}\label{eq:dimension}
         \dim \pi_{I\cap J}(\Gamma)+n\leq \dim \pi_I(\Gamma)+\dim \pi_J(\Gamma).
     \end{equation}
     Since $I\cup J=[k]$, we have $\ker\pi_I\cap\ker\pi_J=\{0\}$.
     Let $P$ be any facet of $\Gamma$. Then $$\ker\pi_I|_{\Lin(P)}\oplus\ker\pi_J|_{\Lin(P)}\subset \ker\pi_{I\cap J}|_{\Lin(P)}.$$
     It follows that
      $$\dim \pi_{I\cap J}(\Lin(P))+n\leq \dim \pi_I(\Lin(P))+\dim \pi_J(\Lin(P)).$$
     Since $\dim\pi_I(P)=\dim\pi_I(\Lin(P))$ for each $I\subset [k]$, we get 
     $$\dim \pi_{I\cap J}(P)+n\leq \dim \pi_I(P)+\dim \pi_J(P).$$
     Taking maximum among all facets of $\Gamma$, we conclude inequality~(\ref{eq:dimension}).
\end{proof}

Proposition~\ref{prop:submodular} also implies that every translation-admissible tropical variety $\Gamma$ in $\BR^k$ gives a usual matroid on $[k]$ with the rank function given by the dimensions of projections of $\Gamma$.

To end the paper, we apply Theorem~\ref{thm:positivity tropical cycle} to give a tropical proof of Theorem~\ref{thm:castillo}.  Let $Y=\prod_{i\in [k]} Y_i$ be a product of smooth projective varieties. Suppose each $Y_i$ is embedded into $\BP^{m_i}$ and let $H_i\subset Y_i$ be the hyperplane class on $Y_i$. For each $I\subset [k]$, let $p_I\colon \prod_{i\in [k]} \BP^{n_i}\rightarrow\prod_{i\in I} \BP^{n_i}$ be the projection map. Write $p_i=p_{\{i\}}$ for $i\in [k]$ for simplicity. Then Theorem~\ref{thm:castillo} is a special case of the following corollary.

\begin{cor}\label{cor:positivity in the algebraic case}
   Let $Y$ be as above and $X$ an (irreducible) closed subvariety of $Y$ of dimension $n$. Then $$X\cdot (p_1^*H_1)^{n_1}\cdots (p_k^*H_k)^{n_k}>0$$ if and only if $\dim p_I(X)\geq n_I$ for any $I\subset [k]$.
\end{cor}
\begin{proof}
  Let $H_i'$ be the hyperplane class in $\BP^{m_i}$. Since the degree of the intersection product in the statement is the same as the degree of the intersection product of $X$ with each $(p_i^*H'_i)^{n_i}$ inside $\prod_{i\in [k]} \BP^{m_i}$, we may just assume that $Y=\prod_{i\in [k]} \BP^{m_i}$ and $H_i=H'_i$.
  
   Let $T_i$ be the maximal torus of each $\BP^{m_i}$, and $T$ the maximal torus of $Y$, which is the product of all $T_i$. We choose a suitable representative of each $H_i$ that intersects transversely with the torus orbits of  $\BP^{m_i}$.  Let $H^\circ_i=H_i\cap T_i$. By Corollary~\ref{cor:tropical hyperplane positive}, $\trop (H_i^\circ)$ is positive in $\BR^{m_i}$.

   Let $t_{i,j}\in T_i$ be general points such that $\trop(t_{i,j})$ is general in $\BR^{m_i}$, where $j\in[n_i]$. This is possible since the points in $\mathbb G_m$ with fixed valuation are Zariski dense. Since $H_i$ intersects transversely with the torus orbits of $\BP^{m_i}$, by the transversality of a general translate (\cite[Corollary 4]{kleiman1974transversality}), the intersection of all $t_{i,j} H_i$s is proper and hence  transverse. As a result, the intersection cycle $(p_1^*H_1)^{n_1}\cdots (p_k^*H_k)^{n_k}$ is represented by $$Z:=\bigcap_{i\in[k]}\bigcap_{j\in[n_i]}p_i^{-1}(t_{i,j}H_i).$$
   Moreover, the intersection of $X$ with  $Z$
   is proper and disjoint with the toric boundary of $Y$.
   Similarly,  the tropicalizations  $$\trop(t_{i,j} H^\circ_i)=\trop(t_{i,j})+\trop(H^\circ_i)$$ also intersect transversely. 

   Let $Z^\circ=Z\cap T$ and $X^\circ=X\cap T$. By the transversality of the tropical intersections discussed above, $\trop(Z^\circ)$ is just the stable intersection of all $\pi_{i}^{-1}(\trop(t_{i,j} H^\circ_i))$ according to Theorem~\ref{thm:lifting tropical intersections}, which is bounded rationally equivalent to the direct product $\prod_{i\in[k]}\trop(H_i^\circ)^{n_i}$. 
   By the same theorem,  we have 
   $$\deg \big(X\cdot Z\big)=\deg \big(X^\circ\cdot Z^\circ\big)=\deg \big(\trop(X^\circ)\cdot\trop(Z^\circ)\big)=\deg \big(\trop(X^\circ)\cdot \prod_{i\in[k]}\trop(H_i^\circ)^{n_i}\big).$$
   By Theorem~\ref{thm:positivity tropical cycle}, this is positive if and only if $\dim\pi_I(\trop(X^\circ))\geq n_I$ for each $I\subset [k]$, which is equivalent to $\dim p_I(X)=\dim p_I(X^\circ)\geq n_I$ for each $I\subset [k]$.
\end{proof}




 \bibliographystyle{abbrv}
 \bibliography{1}

\end{document}